\documentclass{article}
\usepackage{amsmath}
\usepackage{amsthm}
\usepackage{amssymb}
\usepackage{geometry}
\usepackage{xcolor}
\usepackage{authblk}
\usepackage{hyperref}
\usepackage{algorithm2e}
\usepackage{verbatim}
\usepackage{graphicx}
\usepackage{pgfplots}
\usepackage{pgfplotstable}
\pgfplotsset{compat=1.18}
\usepackage{tikz}

\usepackage[giveninits,doi=false,isbn=false,url=false,eprint=false,maxnames=50,sorting=nyt]{biblatex}
\title{Shape optimisation for adaptive $r$-refinement: the one-dimensional case with residual based error estimators}
\author{Philip J.~Herbert}
\affil{
        Department of Mathematics,\\
        University of Sussex,\\
        BN1 9RQ}
\date{}

\newcommand{\dee}{{\rm d}}
\newcommand{\dx}{\dee x}

\newcommand{\ds}{\dee s}

\newcommand{\R}{\mathbb{R}}
\newcommand{\id}{{\rm id}}

\newcommand{\jump}[1]{ \left[\!\left[#1 \right]\!\right]}

\DeclareMathOperator{\Div}{div}
\DeclareMathOperator{\diam}{diam}
\DeclareMathOperator*{\argmin}{argmin}

\newtheorem{lemma}{Lemma}

\newtheorem{thm}{Theorem}

\usepackage{pgfplotstable}
\pgfplotstableread[col sep=comma, trim cells=true]{tikz/twoD.csv}\twoDExactTable
\pgfplotstablegetelem{2}{r_ref}\of\twoDExactTable
\pgfmathparse{sqrt(\pgfplotsretval)}
\edef\twoDExactFinal{\pgfmathresult}

\pgfplotstablegetelem{2}{uniform}\of\twoDExactTable
\pgfmathparse{sqrt(\pgfplotsretval)}
\edef\twoDExactInitial{\pgfmathresult}

\newcommand{\twoDFunction}{\frac{1}{10}(x_1-1)x_1 (x_2-1)x_2 \left(\sum_{k=0}^5 (x_1 + \frac{1}{2})^k \right) \left(\sum_{k=0}^5 (x_2 + \frac{1}{2})^k\right)}

\begin{document}
\maketitle
\begin{abstract}
    We consider the adaptive $r$-refinement for solving the finite element discretisation of a Poisson problem.
    The goal of $r$-refinement is to reposition the nodes of a computational mesh in order to better approximate the finite element problem.
    Since the mesh is being moved, it naturally becomes linked to shape optimisation methods.
    We utilise a standard optimisation algorithm to perform an adaptive $r$-refinement procedure; we demonstrate that this shape optimisation algorithm generates a sequences of meshes which converge.
    The convergence of this algorithm will be shown with a direct calculation of the error, but the most novel aspect of this work is to utilise a standard residual error estimator in one dimension.
    To illustrate the approach, a number of numerical experiments are presented, which verify the efficacy of the method.
\end{abstract}
\section{Introduction}
    The accurate and efficient discretisation of partial differential equations (PDE) is essential in many scientific applications.
    Typically the PDE will be discretised by making use of the finite element method (FEM).
    Here we are introducing a new strategy for adaptive $r$-refinement in the finite element solution for PDEs.
    This framework will make use of recent developments in the computational theory of shape optimisation and apply them to \emph{a posteriori} estimators for the PDE.
    The purpose of $r$-refinement is to move the vertices of the computational mesh which underlies the FEM in order to improve the quality of the discretisation.
    This is often referred to as \emph{mesh moving methods}, which are discussed within \cite{HuaRus11} as well as \cite{BudHuaRus09}.
    The movement of mesh nodes does not change the number of degrees of freedom in contrast to adaptive $h$-refinement, whereby the mesh is locally refined, or adaptive $p$-refinement, where the polynomial degree is locally increased.
    
    Both adaptive $h$- and $p$-refinement methods are better understood than $r$-refinement in practice, indeed they are implemented in many finite element libraries.
    We refer to \cite{BonCanNoc24} for a summary of work on adaptive methods in $h$-refinement, \cite{CanNocSte19} for $p$-refinement, and \cite{CanNocSte16} for $hp$-refinement.
    One important feature of $r$-refinement is the fact that the mesh topology remains fixed.
    By having a fixed topology, this may allow for more proficient use of highly parallel machines when solving the discrete system.

    In general, when given multiple meshes it is not easy to tell which should give a better approximation of the solution of a PDE.
    Unless strictly necessary, it is typical to have a mesh with elements which are as uniform and \emph{regular} as possible.
    With these meshes, classical estimates depend only on $h$, the diameter of the largest cell of the mesh.
    Rather than only the quantity $h$, it is the mesh as a whole which determines the approximation quality.
    It may be the case that a non-uniform mesh with elements which have small angles may provide a better error while having the same mesh topology.
    With this in mind, considering how the mesh can change the error, this suggests it is feasible to view the mesh itself as a variable to be optimised.
    In Figure \ref{fig:exampleImage} we show two meshes which are used to represent the same known function using globally continuous piecewise linear finite elements.
    The mesh on the right is a transformation of the mesh on the left, according to Algorithm \ref{alg:ideal}.
    This transformation leads to a mesh which better approximates a given function in $H^1_0(\Omega)$ while using the same number of degrees of freedom and mesh topology.
    \begin{figure}
        \includegraphics[width = .45 \linewidth]{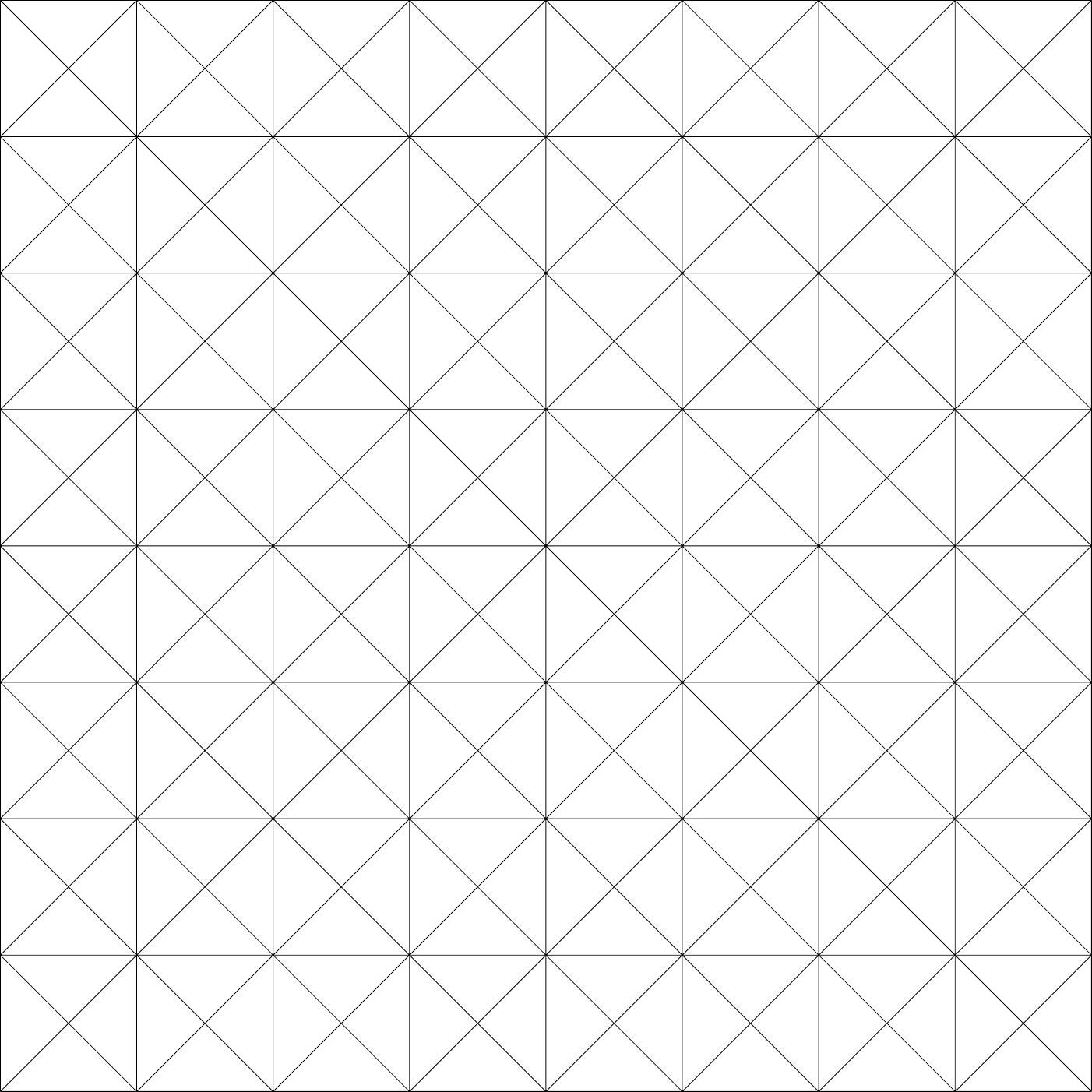}
        \hfill
        \includegraphics[width = .45 \linewidth]{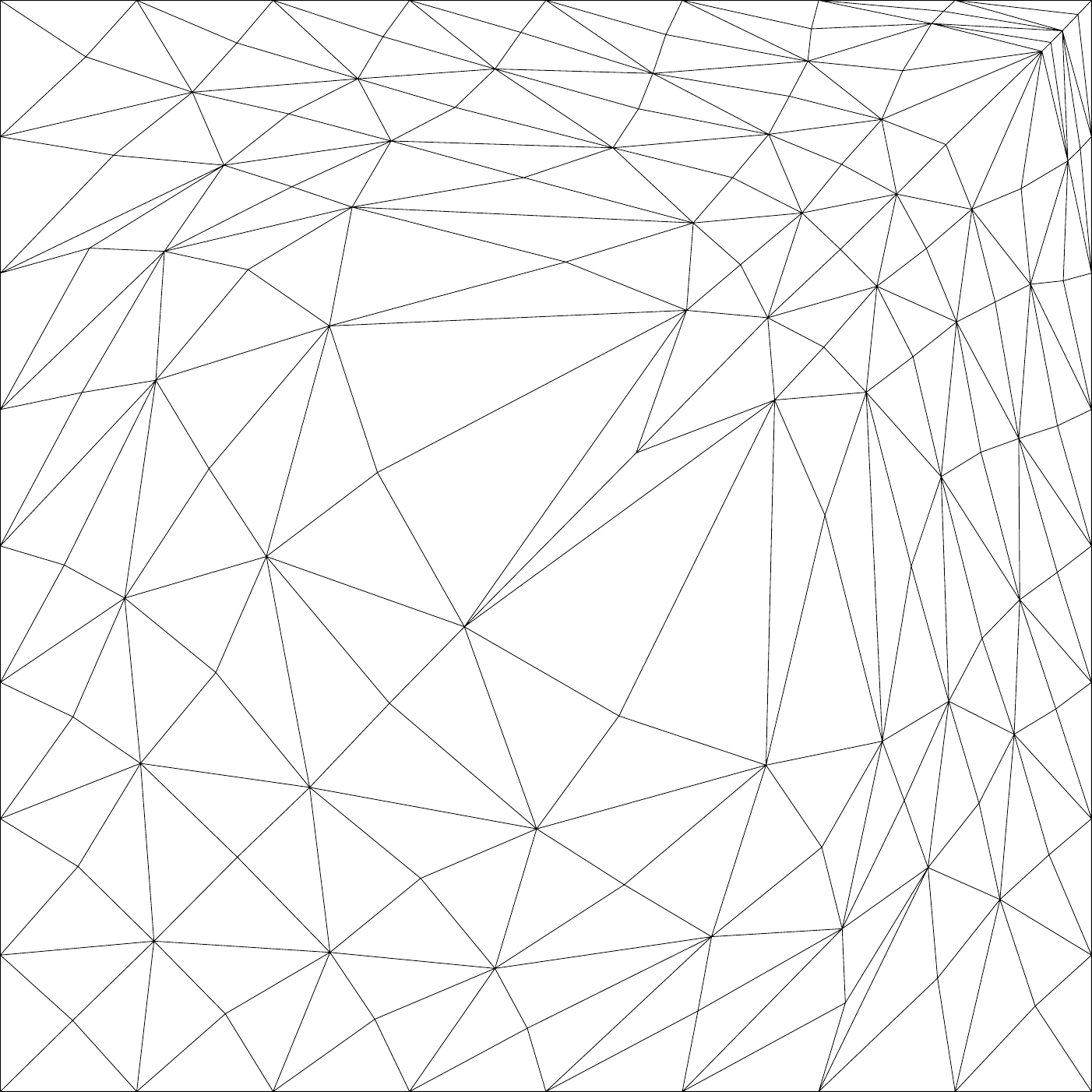}
        \caption{Two meshes which cover $\Omega = (0,1)^2$, upon which the $H^1_0(\Omega)$-Ritz projection of $y = \twoDFunction$ is approximated.
                On the left mesh, the error is given by $\twoDExactInitial$ and on the right $\twoDExactFinal$.
                The mesh on the right better approximates $y$ in $H^1_0(\Omega)$.
                }\label{fig:exampleImage}
    \end{figure}

    The implementation and examination of $r$-refinement methods have gathered attention in recent years; below we discuss some of the literature.
    The work \cite{CaoHuaRus03} covers a variety of approaches to generate mesh moving strategies.
    Using neural networks, \cite{ZhaWanKra24} work to optimise meshes for general PDEs, whereas \cite{AbaFueIli25} uses $r$-refinement to better approximate self-adjoint elliptic problems.
    In \cite{AppBudPry24}, the authors explain how to adapt the mesh to handle elliptic problems on non-convex domains, where there is suboptimal convergence due to the locally poor regularity of the solution.
    One aspect which $r$-refinement can be be applied to is time-dependent problems, \cite{Zeg98} studies both parabolic and hyperbolic equations making use of monitor functions.
    It may be the case that this time-dependent problem has an interface or wavefront to track which is studied in \cite{DonYanHua25} which considers shock-dominated viscous flow problems, but also in \cite{StoMacRus01} which addresses hyperbolic conservation laws.
    The work of \cite{AppBudPry22} considers a moving mesh strategy which also has a mass conservation property for solving the linear transport equation in two dimensions.

    This work will introduce and exploit the fact that there are links between certain strategies in $r$-refinement and methods for shape optimisation.
    As will be seen below shape optimisation will often work with bi-Lipschitz maps from a reference shape.
    In a computational setting, these ensure that the image of a mesh is a mesh and does not become degenerate.
    Hence, restricting the class of maps to be endomorphic means that one is simply reparametrising from a given mesh of the reference shape.
    As such, one may choose as a cost functional as a notion of the error according to the PDE discretisation.
    
    In recent years, the area of computational shape optimisation has bloomed with many different strategies and results.
    A survey of computational shape optimisation may be found in \cite{ADJ21}.
    We will make use of this link with shape optimisation, particularly the work of \cite{DecHerHin25} which provides a methodology to obtain convergence in a nice sense to computational shape optimisation problems.
    This strategy makes use of a $W^{1,\infty}$ method to map the mesh in an appropriate way which we will partially adopt here.

    \paragraph{Main contribution}
    We formulate adaptive r-refinement as an optimisation problem over admissible mesh parametrisations, thereby placing mesh movement within a computational shape optimisation framework.
    With this, we have Theorem \ref{thm:ideal} which proves that one can find a sequences of meshes for which the true finite element error is minimised using an Armijo-type descent method.
    In a fully practical setting, Theorem \ref{thm:practical} establishes convergence of a fully computable residual-based adaptive $r$-refinement strategy in one dimension, where we work to show differentiability of the chosen estimator.
    Numerical evidence demonstrates that the method improves finite element approximations and compares favourably with standard adaptive $h$-refinement approaches.
    
    \paragraph{Outline}
    In Section \ref{sec:prelim} the basic numerical setting is introduced, as well as appropriate notation from shape optimisation.
    This is followed by Section \ref{sec:movingMesh} which provides analytical framework in order to consider the numerical numerical algorithm of interest.
    This algorithm is shown to converge in arbitrary dimensions when the objective to be minimised is the error as well as in a one-dimensional setting when the objective is taken as a fully practical residual-estimator.
    Finally, in Section \ref{sec:exp}, numerical experiments are provided which demonstrate the algorithm's performance in a few interesting settings.
\section{Preliminaries}\label{sec:prelim}
    Here we introduce the problem of interest in this work.
    For simplicity, we consider a Poisson problem with Dirichlet boundary conditions.
    Let $\Omega$ be an open, bounded, and polygonal domain in $\R^d$, let $f \in L^2(\Omega)$ be some data, we then seek
    \begin{equation}\label{eq:Poisson}
        y \in H^1_0(\Omega) \mbox{ such that }
        \int_\Omega \nabla y \cdot \nabla v \dx
        =
        \int_\Omega f v \dx
        \mbox{ for all}
    \end{equation}
    for all $v \in H^1_0(\Omega)$.
    It is well-known that there is a unique such $y \in H_0^1(\Omega)$.
    This problem serves as a model elliptic PDE, and the framework extends naturally to more general elliptic problems with sufficiently regular coefficients.

    When it comes to discretising \eqref{eq:Poisson}, there are a number of approaches, such as finite difference or finite volume; these are not studied here, however it may be interesting to apply techniques from this work to such settings.
    We will consider only a (conforming) finite element discretisation.
    \subsection{Finite element method and meshes}\label{sec:prelim:FEM}
        Let $\mathcal{T}_h$ be a triangulation of $\Omega$, that is a collection of simplices whose union covers $\bar{\Omega}$ and whose pairwise intersection consists of measure zero sets.
        The triangulation $\mathcal{T}_h$ is indexed by $h : = \max_{T \in \mathcal{T}} h_T$ where $h_T = \diam(T)$.
        It is typically assumed that the triangulation $\mathcal{T}_h$ is shape regular in the sense that $\rho_T$, the radius of the largest inscribed ball within $T$, is is uniformly bounded below with respect to $h_T$ on each triangle $T$, this uniformity should hold in the case that the parameter $h$ goes to zero.
        
        For some non-zero natural number $q$, define the finite element space $V_h$ to be
        \begin{equation}
            V_h := \{ \varphi \in C^0(\bar{\Omega}) : \varphi|_{\partial\Omega} = 0,\, \varphi|_T \in \mathbb{P}^q(T)\ \forall T \in \mathcal{T}_h\},
        \end{equation}
        where $\mathbb{P}^q(\omega)$ denotes the space of polynomials of order $q$ on $\omega$.
        It holds that $V_h \subset H_0^1(\Omega)$ and that elements $v \in H_0^1(\Omega)$ may be well approximated by elements of $V_h$.

        The finite element approximation of $y\in H_0^1(\Omega)$ which solves \eqref{eq:Poisson} is given by the unique $y_h \in V_h$ such that
        \begin{equation}
            \int_\Omega \nabla y_h \cdot \nabla v_h \dx
            =
            \int_\Omega f v_h \dx
        \end{equation}
        for all $v_h \in V_h$.
        It is again well-known that if $f \in L^2(\Omega)$ and $\Omega$ is convex one may estimate
        \begin{equation}\label{eq:apriori}
            \|y - y_h\|_{L^2(\Omega)}
            +
            h \| \nabla (y - y_h)\|_{L^2(\Omega)}
            \leq
            C h^2 \|f\|_{L^2(\Omega)}
        \end{equation}
        where $C>0$ is a constant which depends only on the domain $\Omega$, the order of the polynomial approximation space, and shape regularity of the family of triangulations $\mathcal{T}_h$, e.g., \cite{BreSco08}.
        Of course, if $f$ is more regular, one may obtain higher rates of convergence when appropriately high order elements are used.
        
        \subsubsection{Estimating error}
            The estimate \eqref{eq:apriori} is known as an \emph{a priori} estimate, feasibly it could be used to in order to estimate the error in meshes, however this estimate is typically overestimating the error itself.
            In an ideal setting, one might simply calculate $\|\nabla(y-y_h)\|_{L^2(\Omega)}$, however typically one does not have access to the exact solution $y$, as such a middle ground is necessary.
            This leads to what are generally known as \emph{a posteriori} estimators.

            The most simple a posteriori estimator is known as the residual estimator and it is given by
            \begin{equation}\label{eq:estimator}
                E(y_h)
                :=
                \left(\sum_{T \in \mathcal{T}} \left( h_T^2 \int_{T}\left( \Delta y_h + f \right)^2\dx + \frac{h_T}{2} \int_{\partial T} \jump{\nabla y_h}^2 \ds \right) \right)^{\frac{1}{2}}
            \end{equation}
            where $\jump{\cdot}$ is the jump over the $(d-1)$-dimensional face.
            One may show that there is $C>0$ such that
            \begin{equation}\label{eq:reliable}
                \| \nabla (y - y_h) \|_{L^2(\Omega)}^2
                \leq
                C E(y_h)
            \end{equation}
            where $C$ depends on the mesh quality, e.g, \cite[Theorem 9.2.15]{BreSco08}.
            In special cases, e.g., $f$ is polynomial, one may show a reverse estimate: there is $C>0$ such that
            \begin{equation}\label{eq:efficient}
                C \| \nabla (y - y_h) \|_{L^2(\Omega)}^2
                \geq
                E(y_h)
            \end{equation}
            where $C$ again depends on the mesh quality, c.f.~\cite[Theorem 9.3.7]{BreSco08}.

            It is this estimator, \eqref{eq:estimator}, which we will use in a limited setting to practically minimise the error.
            One may consider a variety of other estimators, in particular dual weighted estimators may be of interest, particularly when $y$ is part of an optimal control problem.
            Other choices of estimators are left to future work.

        \subsection{Shape optimisation}\label{sec:prelim:ShapeOpt}
        While it may not be the case that we are optimising a shape, we are interested in moving the interior mesh nodes, which do not necessarily define a shape; many shape optimisation approaches try to ensure that the interior mesh nodes are \emph{well behaved} while the boundary nodes perform the \emph{shape} aspect of shape optimisation.
        It is however the case that, in many computational settings, shape optimisation using finite elements must handle the interior of the domain.
        This is typically done by the method of mappings, in our setting, since the boundary is prescribed, this effectively is a re-parametrisation.
        Let us briefly discuss the shape optimisation concepts which are necessary here, these are adapted from \cite{DecHerHin25}.
        
        \paragraph{Spaces with parametrised meshes}
        We consider the domain $\Omega$ to be Lipschitz and polygonal.
        Let $\hat{\mathcal{T}}$ be a reference triangulation of $\Omega$ which should be shape regular.
        On this triangulation, we define
        \begin{equation}
            \mathcal{U}
            :=
            \{
                \Phi \in C^0(\bar{\Omega};\R^d) : \Phi|_{\hat {T}} \in \mathbb{P}^1(\hat{T};\R^d)\ \forall \hat{T} \in \mathcal{T},\, \Phi \mbox{ is injective},\, \Phi = \id \mbox{ on } \partial \Omega
            \}.
        \end{equation}
        This space $\mathcal{U}$ is the space of parametrisations, we note that it has a group structure where the operation is composition.
        It is worth noting that works such as \cite{PaWeFa18} consider higher order transformations.
        One may seek to allow for a larger space than $\mathcal{U}$, for which the boundary may be reparametrised, however this is not considered at this time.
        In order to perturb elements of $\mathcal{U}$, it is useful to introduce a tangent-like space at each $\Phi \in \mathcal{U}$, this is given by
        \begin{equation}
            \mathcal{V}_\Phi
            :=
            \{
                U \in C^0(\bar{\Omega}; \R^d) : U|_{\partial \Omega} = 0,\, U|_T \in \mathbb{P}^1(T;\R^d), \forall T \in \mathcal{T}_\Phi
            \}.
        \end{equation}

        In \cite[Lemma 2.1]{DecHerHin25} it is shown that the elements of $\mathcal{U}$ are also bilipschitz maps $\bar{\Omega}$ onto $\bar{\Omega}$.
        As such, for any $\Phi \in \mathcal{U}$ one defines the image triangulation as
        \begin{equation}
            \mathcal{T}_{\Phi}
            :=
            \{
                T = \Phi(\hat{T}) : \hat T \in \hat{\mathcal{T}}
            \}.
        \end{equation}
        On this image triangulation, one may further define a finite element space,
        \begin{equation}
            V_\Phi
            :=
            \{
                v_h \in C^0(\bar \Omega) : v_h|_{\partial\Omega} = 0,\, v_h|_{T} \in \mathbb{P}^p(T)\ \forall T \in \mathcal{T}_\Phi
            \}.
        \end{equation}
        We note that for any $\Phi_1,\, \Phi_2 \in \mathcal{U}$ and any $v_{\Phi_1} \in V_{\Phi_1}$, it holds that $v_{\Phi_1} \circ \Phi_1 \circ \Phi_2^{-1} \in V_{\Phi_2}$.
        
        Now that the finite element spaces are defined, one can discuss the discrete solution: define $y_\Phi \in V_\Phi$ to be the unique element such that
        \begin{equation}\label{eq:PoissonDiscrete}
            \int_\Omega \nabla y_\Phi \cdot \nabla v_\Phi \dx
            =
            \int_\Omega f v_\Phi \dx
        \end{equation}
        for all $v_\Phi \in V_\Phi$.
        
        \paragraph{Shape functionals}
        In shape optimisation, one typically seeks to minimise a functional $J$ -- which depends on the shape -- over a space of shapes.
        In the setting of mesh movement, this implicitly incorporates the parametrisation.
        The problem considered in \cite{DecHerHin25} is easily adapted into: find $\Phi^*$ which minimises the map
        \begin{equation}
            \Phi \mapsto J(\Phi) := \int_\Omega j(x, y(x), \nabla y (x)) \dx
        \end{equation}
        where $y = y_\Phi\in V_\Phi$ is the solution to \eqref{eq:PoissonDiscrete} and $j \colon \R^d \times \R \times \R^d \to \R$ satisfies some assumptions.
        To minimise this in practice, one uses first-order methods, this requires the notion of shape derivatives.
        The \emph{shape derivative} of $J$ at $\Phi$ in direction $U \in \mathcal{V}_\Phi$ may be formulated as the usual directional derivative of $J$ in direction $U$ -- by the method of mappings, one sees that the parametrisation can somehow be identified with the shape.
        In the continuous setting, one may require $U \in W^{1,\infty}(\Omega;\R^d)$ to make sense of the shape derivative; in a discrete approach, practically one may use an approach in a larger Hilbert space, e.g., $H^1(\Omega;\R^d)$.
        The reader is referred to \cite{DecHerHin24, DecHerHin21} for further comments on the comparison of the Hilbertian setting to methods using $W^{1,\infty}$, and \cite{ADJ21} is highlighted as an article which discusses many aspects of computational shape optimisation.

        Supposing additional regularity from $f$ and that $j$ is $C^1$, it holds that the derivative $J'(\Phi)[U]$ is given by
        \begin{equation}\begin{split}
            J'(\Phi)[U]
            =&
            \int_\Omega \left( j(\cdot, y_\Phi, \nabla y_\Phi)\Div U + j_x(\cdot, y_\Phi, \nabla y_\Phi) \cdot U - j_z(\cdot, y_\Phi, \nabla y_\Phi) \cdot DU^T \nabla y_\Phi \right)\dx
            \\
            &+
            \int_\Omega \left( \left(\Div U I - DU^T - DU \right)\nabla y_\Phi \cdot \nabla p_\Phi - \Div(f U)p_\Phi \right)\dx
        \end{split}\end{equation}
        where $p_\Phi \in V_\Phi$ is the unique element such that
        \begin{equation}
            \int_\Omega \nabla p_\Phi \cdot \nabla v_\Phi \dx
            =
            - \int_\Omega \left( j_y(\cdot, y_\Phi, \nabla y_\Phi)v_\Phi + j_z(\cdot,y_\Phi,\nabla y_\Phi) \cdot \nabla v_\Phi \right)\dx
        \end{equation}
        for all $v_\Phi \in V_\Phi$, the so-called adjoint.

        With this, we are now prepared to discuss the main aspect of this manuscript.
\section{Minimising the error due to the mesh}\label{sec:movingMesh}
    Since the solution of the finite element problem inherently depends on the mesh, one may seek to ensure that the mesh is performing as well as it can -- or at least, iteratively improve the performance.
    \subsection{The algorithm of interest}\label{sec:movingMesh:ideal}
        A standard method for iteratively minimising energies, is an Armijo-Goldstein method.
        This provides a step size for the first order method to ensure that that cost functional decreases by a certain amount depending on the derivative at the current position.
        This is provided in Algorithm \ref{alg:ideal}.
        \begin{algorithm}
            \caption{A standard Armijo method for minimisation of shape functionals.}
            \label{alg:ideal}
            \KwData{$\hat{\mathcal{T}}$ $J$, $\epsilon>0$, $\gamma \in (0,1)$}
            \KwResult{$\mathcal{T}_{\Phi_n}$}
            Set $\Phi_0 = \id$\;
            \For{$n = 0, 1, 2, \ldots$}{
                Solve for $y_{\Phi_n}$\;
                Calculate $J'(\Phi_n)$\;
                Find $ U_n \in \argmin\{ J'(\Phi_n)[U] : U \in \mathcal{V}_{\Phi_n},\, |DU| \leq 1 \}$\;
                \If{$\|J'(\Phi_n)\| = -J'(\Phi_n)[U_n] \leq \epsilon$}{
                    Return $\mathcal{T}_{\Phi_n}$\;
                }
                Find $\alpha_n = \max \{ 2^{-k} : k \in 1, 2, \dots \mbox{ such that } J(\Phi_n+2^{-k} U_n\circ \Phi_n) - J(\Phi_n) < \gamma 2^{-k} J'(\Phi_n)[U_n] \}$  (requires solving for $y_{(\Phi_n + 2^{-k}U_n\circ \Phi_n) }$)\;
                Set $\Phi_{n+1} = \Phi_n + \alpha_n U_n\circ\Phi_n$\;
            }
        \end{algorithm}

        The first possible example of an appropriate functional $J$ would be to consider $J(\Phi) = \|y-y_\Phi\|_H^2$ where $H = L^2(\Omega)$ or $H= H^1_0(\Omega)$ and $y\in H^1_0(\Omega)$ is the exact solution to \eqref{eq:Poisson}.
        To begin with, let us first show that it is possible to find an admissible step size $\alpha_n$ for Algorithm \ref{alg:ideal} in this setting.
        \begin{lemma}\label{lem:ExactAdmissible}
            Let the functional $J$ be either $J(\Phi):= \int_{\Omega} (y-y_\Phi)^2 \dx$ or $\int_{\Omega} |\nabla(y-y_\Phi)|^2 \dx$ and let $U \in \argmin\{ J'(\Phi_n)[U] : U \in \mathcal{V}_{\Phi},\, |DU| \leq 1\}$, $\gamma \in (0,1)$, and $\epsilon>0$.
            If $J'(\Phi)[U] < -\epsilon$, then there exists $\delta\in (0,1)$, depending only on $f$, $\gamma$, $\Omega$, $d$, $\epsilon$ and the form of the integrand of $J$ such that
            \begin{equation}
                J(\Phi+t U\circ \Phi) - J(\Phi) \leq \gamma t J'(\Phi)[U] 
            \end{equation}
            for all $t \in [0,\delta]$.
        \end{lemma}
        This follows by applying \cite[Lemma 3.2]{DecHerHin25} with the specific choice of $j(X,Y,Z) = (y(X)-Y)^2$ or $j(X,Y,Z) = |\nabla y(X) - Z|^2$ which both satisfy all relevant assumptions.
        
        Similarly, the following result may be given by considering the parts of the result Theorem 3.3 of \cite{DecHerHin25}.
        \begin{thm}\label{thm:ideal}
            Let $J(\Phi):= \int_{\Omega} (y-y_\Phi)^2 \dx$ or $\int_{\Omega} |\nabla(y-y_\Phi)|^2 \dx$, then Algorithm \ref{alg:ideal} converges.
            In particular, the sequence $\|J'(\Phi_n)\|\to 0$ as $n \to \infty$.
            Furthermore, there is a subsequence $\{\Phi_{n_j}\}_{j \in \mathbb{N}}$ and limit $\Phi^*$ with
            \begin{equation}
                \Phi^*\in \{ \Phi \in C^0(\Omega) : \Phi|_{\partial\Omega} = \id,\, \Phi|_{\hat{T}}\in P^1(\hat{T};\R^d)\ \forall \hat{T} \in \hat{\mathcal{T}}\}
            \end{equation}
            such that $\Phi_{n_j} \to \Phi^*$ as $j \to \infty$.
        \end{thm}
        Let us note that it need not be the case that $\Phi^*$ is in $\mathcal{U}$, as one may lack injectivity in the limit.
        Under a condition that the sequence remains under control, say, $\sup_{n \in \mathbb{N}}|D\Phi_n^{-1}|<C$, one has that $\Phi^* \in \mathcal{U}$ and $J(\Phi^*)$ is stationary.
        For the setting considered here, it would perhaps be rather unusual for the limit $\Phi_n$ to head towards a non-injective function as that would be effectively reducing the degrees of freedom available, and convention suggests that more degrees of freedom reduces the error.
    \subsection{A fully practical algorithm in one dimension}\label{sec:movingMesh:practical}
        To make a fully practical algorithm, one wishes to approximate the error by using a posteriori estimators such as the residual as defined in \eqref{eq:estimator}.
        It is not immediately clear that the estimator is differentiable due to the occurrence of $h_T = \sup_{x,y \in T}|x-y|$ which is only Lipschitz over $\mathcal{U}$.
        Because of this, the proposed Algorithm \ref{alg:ideal} may not even make sense; of course one may try other minimisation algorithms which require weaker notions of differentiability, however we here consider a simplified problem and choose to consider the one-dimensional case.

        There are two convenient facts about the one-dimensional setting:
        Firstly, it holds that $h_T = \int_T 1 \dx$, which is known to be smooth.
        Secondly, in one dimension, the estimator \eqref{eq:estimator} is simplified, let us define
        \begin{equation}\label{eq:oneDEstimator}
            E(\Phi) := \sum_{T \in \mathcal{T}_\Phi } h_T^2 \int_T (\partial_x^2 y_\Phi + f)^2 \dx.
        \end{equation}
        It holds that there is $C>0$ such that
        \begin{equation}
            C E(\Phi) \geq \|\partial_x (y-y_\Phi)\|_{L^2(\Omega)}^2,
        \end{equation}
        where $C$ depends only on $\Omega$ and polynomial approximation degree -- mesh elements always have shape regularity constant 1 in one-dimension.
        Since this constant is uniform in $\Phi \in \mathcal{U}$, we plan to minimise $\Phi \mapsto E(\Phi)$ in order to attempt to make $\|\partial_x (y-y_\Phi)\|_{L^2(\Omega)}^2$ smaller.

        With this estimator, we note that its form does not allow one to fully utilise the theory presented within \cite{DecHerHin25}.
        In particular, we do not have a shape derivative readily available, furthermore, this is a term with a second derivative appearing.
        Shape derivatives for problems with higher derivatives have appeared in \cite{GraKie21} and \cite{EllHer21} in the context of biomembranes.
        
        It is first useful to recall the following change of variables result, taken from, e.g., \cite{GraKie21}.
        \begin{lemma}
            Let $\Omega_1, \Omega_2\subset \R^d$ be open with $\Psi \colon \Omega_1 \to \Omega_2$ be a $C^2$ diffeomorphism, then $f \in C^2(\Omega_2)$ if and only if $f \circ \Psi \in C^2(\Omega_1)$.
            Furthermore,
            \begin{equation}
                (\Delta f) \circ \Psi
                =
                \frac{1}{\det(D\Psi)} \Div\left( \det(D\Psi)D\Psi^{-1}D\Psi^{-T} \nabla (f\circ \Psi) \right).
            \end{equation}
        \end{lemma}
        In one dimension, this simplifies to $(\partial_x^2 f) \circ \Psi = (\partial_x \Psi )^{-1} \partial_x\left( (\partial_x \Psi)^{-1} \partial_x (f\circ \Psi) \right)$.
        The proof is standard calculus.
        Such a result allows for one to take the derivative of the Laplacian with respect to $\Phi$.
        \begin{lemma}
            Let $\Phi \in \mathcal{U}$, then for each $T \in \mathcal{T}_\Phi$ it holds that
            \begin{equation}\begin{split}
                (\partial_x^2 y_\Phi)'[U]
                :=&\,
                \lim_{t \to 0} \frac{ (\partial_x^2 y_{\Phi+tU\circ \Phi}) \circ (\Phi+tU\circ \Phi) \circ \Phi^{-1} - \partial_x^2 y_\Phi}{t}
                \\=&\,
                -\partial_x U \partial_x^2 y_\Phi + \partial_x\left(-\partial_x U \partial_x y_\Phi \right) + \partial_x^2( y_\Phi'[U])
            \end{split}\end{equation}
            in $T$, where $y_\Phi'[U] := \lim_{t \to 0} \frac{y_{\Phi + t U}\circ (\Phi+tU\circ \Phi) \circ \Phi^{-1} - y_\Phi}{t}$.
        \end{lemma}
        The proof of this follows from calculating $(\partial_x^2 y_{\Phi+tU\circ\Phi}) \circ (\Phi+tU\circ\Phi)$, adding and subtracting appropriate terms and recalling the derivative for the inverse.
        A multi-dimensional version of this appears in \cite{GraKie21}, for example.

        By making use of the Lagrangian framework for derivatives for PDE constrained optimisation, one does not need to worry about calculating the derivative of $y_\Phi$ with respect to $\Phi$ in this smooth context as the adjoint will handle it.
        One may show that the result we calculate is equivalent by an inverse function theorem-type argument.
        
        Now we calculate the derivative of the mesh size:
        For $\Phi \in \mathcal{U}$ and each $T \in \mathcal{T}_\Phi$, it holds by the Leibniz rule that
        \begin{equation}
            h_T'[U] = \int_T \partial_x U \dx.
        \end{equation}
        With this, one has the following result on the form of the derivative of $E$.
        \begin{lemma}
            Let $\Phi \in \mathcal{U}$, then it holds that
            \begin{equation}\begin{split}
                E'(\Phi)[U]
                =&\,
                \int_\Omega \left( -\partial_x U \partial_x y_\Phi \cdot \partial_x p_\Phi - \partial_x(U f)p_\Phi \right) \dx
                +
                2\sum_{T \in \mathcal{T}_\Phi}
                h_T^2 \int_T \partial_x U (\partial_x^2 y_\Phi + f )^2 \dx
                \\
                &+
                2\sum_{T \in \mathcal{T}_\Phi}
                h_T^2 \int_T
                \big(
                        -2 \partial_x U \partial_x^2 y_\Phi
                        + \partial_x(U f)
                \big)(\partial_x^2 y_\Phi + f) \dx
            \end{split}\end{equation}
            for all $U \in \mathcal{V}_\Phi$, where $p_\Phi \in V_{\Phi}$ is the unique element such that
            \begin{equation}
                \int_\Omega \partial_x p_\Phi \cdot \partial_x v_\Phi \dx
                =
                -2 \sum_{T \in \mathcal{T}_\Phi} h_T^2 \int_T (\partial_x^2 y_\Phi + f ) v_\Phi \dx.
            \end{equation}
        \end{lemma}
        \begin{proof}
            The form of the derivative follows from the standard Lagrangian strategy for PDE constrained optimisation, whereby one has the derivative of the state constraint evaluated on the state and adjoint plus the derivative of the energy evaluated on the state.
            The main steps of the proof then rely on applying the product rule to differentiate with respect to the state, as well as making use of the fact that for any $U \in \mathcal{V}_\Phi$, it holds that for any $T \in \mathcal{T}_\Phi$, $\partial_x U$ is constant on $T$, hence $h_T'[U] = \int_T \partial_x U \dx = h_T \partial_x U|_T$.
        \end{proof}

        To ensure that Algorithm \ref{alg:ideal} makes sense, one ensures that there is an admissible Armijo step.
        \begin{lemma}\label{lem:EstimatorAdmissible}
            Let $\Phi \in \mathcal{U}$, $U \in \argmin\{ E'(\Phi)[U] : U \in \mathcal{V}_{\Phi},\, |\partial_x U| \leq 1\}$, $\gamma \in (0,1)$, and $\epsilon >0$.
            If $J'(\Phi)[U] < -\epsilon$ and $f \in H^2(\Omega)$, then there exists $\delta\in (0,1)$, depending only on $q$, $f$, $\gamma$, $\Omega$, and $\epsilon$ such that
            \begin{equation}
                E(\Phi+t U) - E(\Phi) \leq \gamma t E'(\Phi)[U] 
            \end{equation}
            for all $t \in [0,\delta]$.
        \end{lemma}
        \begin{proof}
            For convenience, let us write $\Psi_t := (\Phi + t U \circ \Phi)\circ \Phi^{-1} = \id + t U$ and also $\hat{y}_t := y_{\Psi_t \circ \Phi} \circ \Psi_t \in V_\Phi$.
            It is useful to recall the inverse estimate,
            \begin{equation}\label{eq:inverseEstimate}
                \|\partial_x v\|_{L^2(T)}
                \leq
                \sqrt{12}(q+1)^2 h_T^{-1} \|v\|_{L^2(T)}
            \end{equation}
            for all $v \in \mathbb{P}^q(T)$, c.f.~\cite{Sch98}.
            From this, one may estimate
            \begin{equation}
                \|\partial_x^2 (\hat{y}_t- y_\Phi) \|_{L^2(T)}
                \leq
                \sqrt{12}q^2h_T^{-1} \|\partial_x (\hat{y}_t- y_\Phi) \|_{L^2(T)}.
            \end{equation}
            Similarly, one may estimate $\|\partial_x^2 \hat{y}_t\|_{L^2(T)}$.
            It is useful to note that, exactly as in \cite[Lemma 3.2]{DecHerHin25}, one has that there is $\delta_1\in (0,1)$ such that $\|\partial_x (\hat{y}_t- y_\Phi) \|_{L^2(T)} \leq c t \|f\|_{H^1(\Omega)}$ for all $t \in (0,\delta_1)$, where $c$ is depending only on $\Omega$.

            We now turn to the object of interest:
            \begin{equation}\begin{split}
                E(\Phi + t U\circ \Phi)
                =
                E(\Psi_t \circ \Phi)
                =
                \sum_{T \in \mathcal{T}_\Phi } \left(
                    \left( \int_T (1+ t \partial_x U) \dx \right)^2 \int_{\Psi_t(T)} \left( \partial_x^2 y_{\Psi_t\circ\Phi}  + f \right)^2 \dx
                \right).
            \end{split}\end{equation}
            We begin by calculating
            \begin{equation}\begin{split}
                \left( \int_T (1+ t \partial_x U) \dx \right)^2
                &=
                h_T^2 + 2 t h_T \int_T \partial_x U \dx + \left(t \int_T \partial_x U \dx \right)^2
                \\
                &=
                h_T^2 + 2 t h_T^2 \partial_x U |_T + h_T^2 t^2 |\partial_x U|_T|^2.
            \end{split}\end{equation}
            For convenience, let $r_1 := h_T^2 t^2 |\partial_x U|_T|^2$, then one has
            \begin{equation}\begin{split}
                \sum_{T \in \mathcal{T}_\Phi} r_1 \int_{\Psi_t(T)} \left( \partial_x^2 y_{\Psi_t\circ\Phi}  + f \right)^2 \dx
                \leq&\
                \sum_{T \in \mathcal{T}_\Phi} 2 r_1 \left( \|\partial_x^2 y_{\Psi_t\circ\Phi}\|_{L^2(\Psi_t(T))}^2 + \|f\|_{L^2(\Psi_t({T}))}^2 \right)
                \\
                \leq&\
                \sum_{T \in \mathcal{T}_\Phi} 2 r_1 \left( 12 q^4 h_{\Psi_t(T)}^{-2} \|\partial_x y_{\Psi_t\circ\Phi}\|_{L^2(\Psi_t(T))}^2 + \|f\|_{L^2(\Psi_t({T}))}^2 \right)
                \\
                \leq&\
                c t^2 q^4 \|f\|_{L^2(\Omega)}^2
            \end{split}\end{equation}
            where we have made use of $|\partial_x U|\leq 1$ and that $h_T (h_{\Psi_t(T)})^{-1} = 1 + r_2$, where $|r_2| \leq c t$.

            From here, the proof will follow along similar lines to \cite{DecHerHin25} on individual elements $T \in \mathcal{T}_\Phi$.
            \begin{equation}\begin{split}
                \int_{\Psi_t(T)} \left( \partial_x^2 y_{\Psi_t\circ\Phi}  + f \right)^2 \dx
                =&\,
                \int_{T} \left( \partial_x^2 y_\Phi + f \right)^2 \dx
                +
                \int_{T} \left( \partial_x^2 y_\Phi + f \right)^2 (\partial_x \Psi_t - 1)\dx
                \\
                &+
                \int_T\left( \left( (\partial_x\Psi)^{-2}\partial_x^2 \hat{y}_t + f \circ \Psi_t \right)^2 - \left( \partial_x^2 y_\Phi + f \right)^2 \right) \dx
                \\
                &+
                \int_T\left( \left( (\partial_x\Psi)^{-2}\partial_x^2 \hat{y}_t + f \circ \Psi_t \right)^2 - \left( \partial_x^2 y_\Phi + f \right)^2 \right) ( \partial_x \Psi_t - 1) \dx
                \\
                =&\,
                \int_{T} \left( \partial_x^2 y_\Phi + f \right)^2 \dx
                +
                I_1 + I_2 + I_3.
            \end{split}\end{equation}
            We immediately have that $I_1 = t \int_{T} \partial_x U ( \partial_x^2 y_\Phi + f )^2 \dx$.
            For $I_2$, we use Taylor's formula to obtain
            \begin{equation}\begin{split}
                \left((\partial_x\Psi)^{-2}\partial_x^2 \hat{y}_t + f \circ \Psi_t \right)^2 -&\, \left( \partial_x^2 y_\Phi + f \right)^2
                \\
                =&\,
                2 \left( \partial_x^2 y_\Phi + f \right) t U \partial_x f 
                +
                2 \left( \partial_x^2 y_\Phi + f \right) ( (\partial_x \Phi_t)^{-2}\partial_x^2 \hat{y}_t - \partial_x^2 y_\Phi )
                \\
                &+
                \int_0^1(1-s) \frac{d^2}{ds^2}\left( \chi_s ^2 \right)\ds
            \end{split}\end{equation}
            where $\chi_s := (1-s)\partial_x^2 y_\Phi + s (\partial_x \Psi_t)^{-2}\partial_x^2 \hat{y}_t  + f \circ \Psi_{st}$.
            Therefore, one has
            \begin{equation}\begin{split}
                I_2
                =&\,
                t \int_T 2 \left( \partial_x^2 y_\Phi + f \right) U \partial_x f \dx
                -
                2 t \int_T 2 \left( \partial_x^2 y_\Phi + f \right) \partial_x U \partial_x^2 y_\Phi \dx
                \\
                &+
                \int_T 2 \left( \partial_x^2 y_\Phi + f \right) ( ((\partial_x \Psi_t)^{-2} - 1 + 2t \partial_x U )\partial_x^2 \hat{y}_t + 2t \partial_x U ( \partial_x^2 y_\Phi - \partial_x^2 \hat{y}_t) ) \dx
                \\
                &+
                \int_T 2 \left( \partial_x^2 y_\Phi + f \right) ( \partial_x^2 y_\Phi - \partial_x^2 \hat{y}_t) \dx
                +
                \int_T \int_0^1(1-s) \frac{d^2}{ds^2}\left( \chi_s^2 \right)\ds \dx
                \\
                =&\,
                t \int_T 2 \left( \partial_x^2 y_\Phi + f \right) U \partial_x f \dx
                -
                2 t \int_T 2 \left( \partial_x^2 y_\Phi + f \right) \partial_x U \partial_x^2 y_\Phi \dx
                +
                I_{2,1} + I_{2,3} + I_{2,3}
            \end{split}\end{equation}
            Making use of $(\partial_x \Psi_t)^{-2} = 1 - 2 t \partial_x U + r_1$ where $|r_1| \leq c t^2$, it follows that
            \begin{equation}\begin{split}
                I_{2,1}
                \leq
                \| \partial_x^2 y_\Phi + f \|_{L^2(T)}\left( c t^2\|\partial_x^2 \hat{y}_t\|_{L^2(T)} + ct\|\partial_x^2 y_\Phi - \partial_x^2 \hat{y}_t \|_{L^2(T)}  \right).
            \end{split}\end{equation}
            For $I_{2,2}$, it holds
            \begin{equation}\begin{split}
                I_{2,2}
                =
                \frac{1}{h_T^2}\int_T \nabla p_\Phi \cdot \nabla (\hat{y}_t - y_\Phi)\dx,
            \end{split}\end{equation}
            hence one may estimate $\sum_{T \in \mathcal{T}_\Phi} h_T^2 I_{2,2}$ identically as in equation (3.16) of \cite{DecHerHin25}, which estimates the corresponding term $T_{2,4}$, to obtain
            \begin{equation}\begin{split}
                \sum_{T \in \mathcal{T}_\Phi} h_T^2 I_{2,2}
                \leq&\,
                2t \int_\Omega \partial_x U \nabla y_\Phi \cdot \nabla p_\Phi \dx - t \int_\Omega \partial_x ( f U ) p_\Phi \dx
                \\
                &+
                ct^2 + c \|\hat{y}_t - y\|_{H^1}^2 + ct \sup_{0\leq \sigma \leq t} \| \nabla f \circ \Psi_t - \nabla f\|_{L^2(\Omega)}.
            \end{split}\end{equation}
            It remains to estimate $I_{2,3}$, which involves the second derivative term, then
            \begin{equation}\begin{split}
                I_{2,3}
                =&\,
                \int_T \int_0^1(1-s) \frac{d^2}{ds^2}(\chi_s^2)\ds \dx
                \\
                =&\,
                \int_T \int_0^1 2 (1-s)
                    \chi_s  t^2\left(\partial_x U \partial_x f \circ \Psi_{st} + U^2 \partial_x^2 f \circ \Psi_{st} \right) \ds\dx
                \\
                &+
                \int_T \int_0^1 2 (1-s) \left[
                    (-\partial_x^2 y_\Phi + (\partial_x \Psi_t)^{-2}\partial_x^2 \hat{y}_t + t U \partial_x f \circ \Psi_{st} )^2 \right]\ds\dx
                \\
                \leq&\,
                c t^2 \|\chi_s \|_{L^2(T)} (\|\partial_x f \circ \Psi_{st}\|_{L^2(T)} + \|\partial_x \Psi_{st}\|_{L^\infty(T)}^2 \|\partial_x^2 f \circ \Psi_{st}\|_{L^2(T)} )
                \\
                &+
                c ( \|\partial_x^2 y_\Phi -\partial_x^2 \hat{y}_t\|_{L^2(T)}^2 + \|(1-(\partial_x \Psi_t)^{-2})\partial_x^2 \hat{y}_t\|_{L^2(T)}^2 + t^2 \|\partial_x \Psi_{st}\|_{L^\infty(T)}^2\|\partial_x f \circ \Psi_{st}\|_{L^2(T)}^2 )
                \\
                \leq&\,
                ct^2 \left( q^4 h_T^{-2} \|\partial_x y_\Phi \|_{L^2(T)} + \|f\|_{L^2(T)} \right)\left( \|\partial_x f\|_{L^2(T)} + \|\partial_x^2 f \|_{L^2(T)} \right)
                +
                c q^4 h_T^{-2} \|\partial_x (y_\Phi - \hat{y}_t)\|_{L^2(T)}^2,
            \end{split}\end{equation}
            where one has used that $\|\chi_s\|_{L^2(T)} \leq c(\|\partial_x^2 y_\Phi\|_{L^2(T)} + \|(\partial_x \Psi_t)^{-2}\partial_x^2 \hat{y}_t\|_{L^2(T)} + \|f\circ \Psi_{st}\|_{L^2(T)} )$.
            Making use of the arguments above, one has a bound on $I_3$ of the form:
            \begin{equation}
                I_3 \leq ct^2 + c h_T^{-2} \|\partial_x(y_\Phi - \hat{y}_t)\|_{L^2(\Omega)}^2.
            \end{equation}

            Summing together all parts, making use of the fact that $h_T \leq |\Omega|$, we have shown there is $c>0$ depending only on the polynomial power of the approximation, the size of $\Omega$ and $f$ such that
            \begin{equation}\begin{split}
                E(\Phi+tU) - E(\Phi)
                \leq&\,
                t E'(\Phi)[U]
                +
                c t \left(t + \sup_{0\leq \sigma \leq t}\|\nabla f\circ \Psi_\sigma - \nabla f \|_{L^2(\Omega)} \right)
                \\
                \leq&\,
                \gamma t E'(\Phi)[U]
                +
                c t \left(t + \sup_{0\leq \sigma \leq t}\|\nabla f\circ \Psi_\sigma - \nabla f \|_{L^2(\Omega)} \right) - (1-\gamma)t \epsilon.
            \end{split}\end{equation}
            This is as in equation (3.21) of \cite{DecHerHin25}, whereby one argues that there exists $\delta\in (0,\delta_1)$ such that $\sup_{0\leq \sigma \leq t}\|\nabla f\circ \Psi_\sigma - \nabla f \|_{L^2(\Omega)} \leq \frac{1}{2c}(1-\gamma)\epsilon$ for $t \in (0,\delta)$, hence one has the existence of a $\delta>0$ such that
            \begin{equation}
                E(\Phi+tU) - E(\Phi)
                \leq
                \gamma t E'(\Phi)[U]
            \end{equation}
            for all $t \in (0,\delta)$, as required.
        \end{proof}
        The proof of the following theorem again follows as in \cite{DecHerHin25} due to the admissibility of the step size.
        \begin{thm}\label{thm:practical}
            Let $J(\Phi):= E(\Phi)$, then Algorithm \ref{alg:ideal} terminates.
            In particular, it holds that the sequence $\|E'(\Phi_n)\|\to 0$ as $n \to \infty$.
            Furthermore, there is a subsequence $\{\Phi_{n_j}\}_{j \in \mathbb{N}}$ and limit $\Phi^*$ with
            \begin{equation}
                \Phi^*\in \{ \Phi \in C^0(\Omega) : \Phi|_{\partial\Omega} = \id,\, \Phi|_{\hat{T}}\in P^1(\hat{T};\R^d)\ \forall \hat{T} \in \hat{\mathcal{T}}\}
            \end{equation}
            such that $\Phi_{n_j} \to \Phi^*$ as $j \to \infty$.
        \end{thm}
        A consequence of the above is that one has a practical algorithm which will run and converge.
        We now turn to numerical experiments to verify as to whether it is useful.

\section{Numerical Experiments}\label{sec:exp}
We now consider some numerical experiments to verify the analytic results.
These are conducted using DUNE \cite{duneReference}, making particular use of the DUNE Python bindings \cite{DunePython1,DunePython2}.
Future work seeks to develop a package in order to simplify the process. 
For the experiments presented, the results appear to hold true for more complicated data also, although the solver parameters required more delicate care in choosing.

For the experiments considered, one always has that the exact solution is known.
Having the exact solution allows one to calculate the exact error, even in the case where one only uses an estimator to attempt to reduce the error.
Throughout the experiments, when we consider a mesh in one dimension, it will cover $(0,1)$ and have uniform elements to begin with.

For the majority of the numerical experiments, we will consider linear finite elements.
All the results remain valid for higher order Lagrange finite elements.
An experiment which uses quadratic finite elements is included; where of course the estimator \eqref{eq:oneDEstimator} has dependence on the local second derivative of the solution.
For the experiments we will run Algorithm \ref{alg:ideal} with the energy $E$ as in \eqref{eq:oneDEstimator}, as well as with $J(\Phi) = \|\nabla(y-y_\Phi)\|_{L^2(\Omega)}^2$ as described in Theorem \ref{thm:ideal}.
The experiment stopped when either $|\mathcal{J}'(\Phi_n)[V_n]| \leq 10^{-5}$ or $n = 20$, and the Armijo constant $\gamma$ is given by $10^{-3}$.
When we run the optimisation with the estimator $E$, we also calculate the exact error $J$ in order to be able to comment on the efficacy of the method.

It is worth noting that finding the update step $U_n$ as described in Algorithm \ref{alg:ideal} is highly non-trivial.
In practice one might wish to find an alternate update step using, say, an Hilbertian method as is discussed in the introduction.
Using such a method will provide a different optimised solution but should equally make the meshes better.
While proving Lemma \ref{lem:EstimatorAdmissible}, it was required to ensure that $|\partial_x U|\leq 1$ which is not guaranteed by generic Hilbertian methods.
In this work, we use the alternating descent method of multipliers (ADMM), to find the update step for $d>1$, this is as described in \cite{DecHerHin24,DecHerHin25}.
When $d=1$, we use \texttt{scipy.optimize.linprog} in Python.

Before we begin the experiments, we remark that one must be careful with the quadrature order when conducting the experiments.
Having a quadrature which is not exact may lead to non-existence of an Armijo step as in Lemmas \ref{lem:ExactAdmissible} and \ref{lem:EstimatorAdmissible}.
Of course, one may not be able to make the quadrature exact, but it was found to be necessary to make the order 'high enough' in order to have admissible steps until the stopping criteria.
Where possible, the \texttt{estimate\_total\_polynomial\_degree} is used from \texttt{ufl} \cite{UFLReference}, which is exact on polynomials and follows some rules for non-polynomial functions.

In the experiments we see below, one does not see a better rate of convergence than that within \eqref{eq:apriori}, this is attributed to the smooth (polynomial) nature of the exact solutions chosen.
Choosing less regular data
\subsection{Convergence of the algorithm using the estimator}\label{sec:exp:conv:res}
    In this section we are in the one dimensional setting.
    The experiments are run using Algorithm \ref{alg:ideal} with the energy being given by the estimator $E$ defined in \eqref{eq:oneDEstimator}.
    Here we consider $f(x) =-C 2x^2(6 - 20 x - 15 x^2)$ for which one has $y(x) = C (x-1)^2 x^4$, where $C = \sqrt{6435}$ is chosen so that $\|\nabla y \|_{L^2((0,1))} = 1$.
    \label{sec:exp:conv:res:lin}
    In Figure \ref{fig:exp:conv:res:meshes:iterates}, one has for a coarse mesh the sequence of mesh iterates along the optimisation path.
    The optimised meshes are plotted in Figure \ref{fig:exp:conv:res:meshes:final}, where it may be noteworthy that there does not appear to be any immediate hierarchical structure to the sequence of optimised meshes.
    \begin{figure}
        \begin{minipage}{0.45\textwidth}\centering
            \begin{tikzpicture}[scale = 6]
                \pgfplotstableread[col sep=comma]{tikz/r_ref/oneD/residual/ref_4/init_mesh.csv}\meshZero
                \draw (0,0)-- (1,0);
                \pgfplotstableforeachcolumnelement{x}\of\meshZero\as\x{\draw (\x,0.08) -- (\x,-0.08);}

                \pgfplotstableread[col sep=comma]{tikz/r_ref/oneD/residual/ref_4/mesh1.csv}\meshOne
                \draw (0,-0.2)-- (1,-0.2);
                \pgfplotstableforeachcolumnelement{x}\of\meshOne\as\x{\draw (\x,-0.12) -- (\x,-0.28);}

                \pgfplotstableread[col sep=comma]{tikz/r_ref/oneD/residual/ref_4/mesh3.csv}\meshTwo
                \draw (0,-0.4)-- (1,-0.4);
                \pgfplotstableforeachcolumnelement{x}\of\meshTwo\as\x{\draw (\x,-0.32) -- (\x,-0.48);}

                \pgfplotstableread[col sep=comma]{tikz/r_ref/oneD/residual/ref_4/mesh7.csv}\meshThree
                \draw (0,-0.6)-- (1,-0.6);
                \pgfplotstableforeachcolumnelement{x}\of\meshThree\as\x{\draw (\x,-0.52) -- (\x,-0.68);}

            \end{tikzpicture}
            \caption{For the mesh with $17$ vertices, from top to bottom we provide the mesh iterates 0, 1, 3, and 7 when using Algorithm \ref{alg:ideal} with estimator \eqref{eq:oneDEstimator} and linear finite elements as described in Section \ref{sec:exp:conv:res:lin}.}\label{fig:exp:conv:res:meshes:iterates}
        \end{minipage}
        \hfill
        \begin{minipage}{0.45\textwidth}\centering
            \begin{tikzpicture}[scale = 6]

                \pgfplotstableread[col sep=comma]{tikz/r_ref/oneD/residual/ref_3/final_mesh.csv}\meshZero
                \draw (0,0)-- (1,0);
                \pgfplotstableforeachcolumnelement{x}\of\meshZero\as\x{\draw (\x,0.08) -- (\x,-0.08);}

                \pgfplotstableread[col sep=comma]{tikz/r_ref/oneD/residual/ref_4/final_mesh.csv}\meshOne
                \draw (0,-0.2)-- (1,-0.2);
                \pgfplotstableforeachcolumnelement{x}\of\meshOne\as\x{\draw (\x,-0.12) -- (\x,-0.28);}

                \pgfplotstableread[col sep=comma]{tikz/r_ref/oneD/residual/ref_5/final_mesh.csv}\meshTwo
                \draw (0,-0.4)-- (1,-0.4);
                \pgfplotstableforeachcolumnelement{x}\of\meshTwo\as\x{\draw (\x,-0.32) -- (\x,-0.48);}

                \pgfplotstableread[col sep=comma]{tikz/r_ref/oneD/residual/ref_6/final_mesh.csv}\meshThree
                \draw (0,-0.6)-- (1,-0.6);
                \pgfplotstableforeachcolumnelement{x}\of\meshThree\as\x{\draw (\x,-0.52) -- (\x,-0.68);}

                \pgfplotstableread[col sep=comma]{tikz/r_ref/oneD/residual/ref_7/final_mesh.csv}\meshFour
                \draw (0,-0.8)-- (1,-0.8);
                \pgfplotstableforeachcolumnelement{x}\of\meshFour\as\x{\draw (\x,-0.72) -- (\x,-0.88);}

            \end{tikzpicture}
            \caption{For a sequence of meshes, the optimised meshes when using Algorithm \ref{alg:ideal} with estimator \eqref{eq:oneDEstimator} and linear finite elements as described in Section \ref{sec:exp:conv:res:lin}.}\label{fig:exp:conv:res:meshes:final}
        \end{minipage}
    \end{figure}

    The resulting errors and estimators are plotted in Figure \ref{fig:exp:conv:res:finalErrors}, where one sees that optimising based on the residual is improving the error of the approximation.
    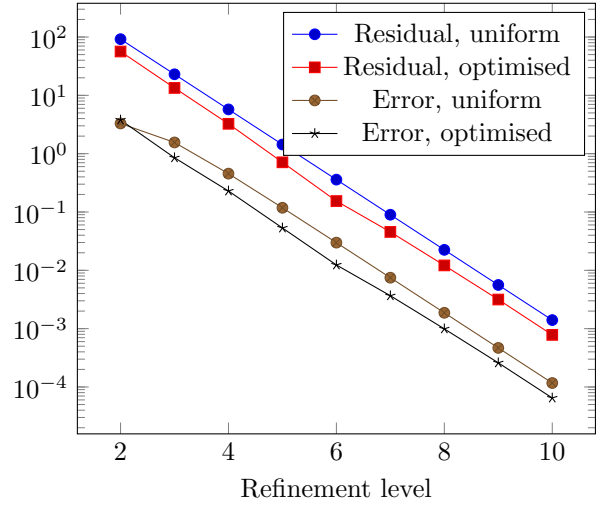
\begin{figure}
        \centering
        \begin{tikzpicture}
            \pgfplotstableread[col sep=comma]{tikz/r_ref/oneD/residual/data.csv}\resData
            \begin{axis}[ymode = log, xlabel = {Refinement level}]
                \addplot table [x=refine, y = initEst]{\resData};
                \addlegendentry{Residual, uniform}
                \addplot table [x=refine, y = finalEst]{\resData};
                \addlegendentry{Residual, optimised}
                \addplot table [x=refine, y = initErr]{\resData};
                \addlegendentry{Error, uniform}
                \addplot table [x=refine, y = finalErr]{\resData};
                \addlegendentry{Error, optimised}
            \end{axis}
        \end{tikzpicture}
            \caption{The graph of errors on the uniform mesh and the optimised mesh as well as the estimators on the uniform mesh and optimised mesh when using Algorithm \ref{alg:ideal} with estimator \eqref{eq:oneDEstimator} and linear finite elements as described in Section \ref{sec:exp:conv:res:lin}.
            Notice that the optimised mesh is always performing better than the uniform mesh, besides in the most coarse mesh.}\label{fig:exp:conv:res:finalErrors}

    \end{figure}

\subsection{Convergence of the algorithm using the \texorpdfstring{$H^1$}{H1} error}\label{sec:exp:conv:ideal}
    Here, only the two dimensional setting is considered.
    The experiments are run using Algorithm \ref{alg:ideal} with the energy being given by $J(\Phi) = \|\nabla(y-y_\Phi)\|_{L^2(\Omega)}^2$ and linear finite elements.
    For this experiment, $\Omega = (0,1)^2$, where $y(x) = \twoDFunction$ and we choose $f(x) = -\Delta y(x)$.
    The resulting errors are plotted in Figure \ref{fig:exp:conv:ideal:twoD:scaling}, where one sees that an improvement is made.
    \begin{figure}
        \centering
        \begin{tikzpicture}
            \pgfplotstableread[col sep=comma]{tikz/twoD.csv}\compareTable
            \begin{axis}[ymode = log, xlabel = {Refinement level}, ylabel={Error}]
                \addplot table [x = refine, y = uniform]{\compareTable};
                \addlegendentry{Uniform mesh}
                \addplot table [x = refine, y = r_ref]{\compareTable};
                \addlegendentry{Optimised mesh}
            \end{axis}
        \end{tikzpicture}
        \caption{Plotted is the errors for the two-dimensional example on the uniform mesh and the optimised mesh when using Algorithm \ref{alg:ideal} with the error and linear finite elements as described in Section \ref{sec:exp:conv:ideal}.
        As expected, the optimised mesh is performing better than the uniform mesh, however it is only doing better by a constant.}\label{fig:exp:conv:ideal:twoD:scaling}
    \end{figure}
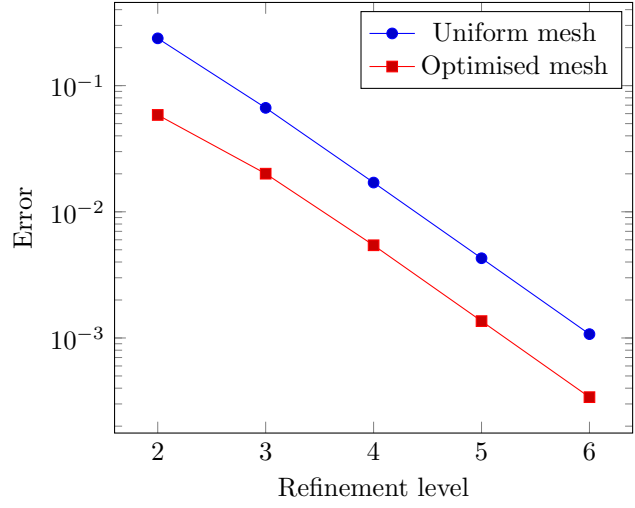
    The optimised meshes are plotted alongside the original meshes in Figure \ref{fig:exp:conv:ideal:twoD:meshes}.
    \begin{figure}
        \centering
        \includegraphics[width = .4 \linewidth]{tikz/r_ref/twoD/ref_3/init_mesh.pdf}
        \hfill
        \includegraphics[width = .4 \linewidth]{tikz/r_ref/twoD/ref_3/final_mesh.pdf}
        \\
        [.25cm]
        \includegraphics[width = .4 \linewidth]{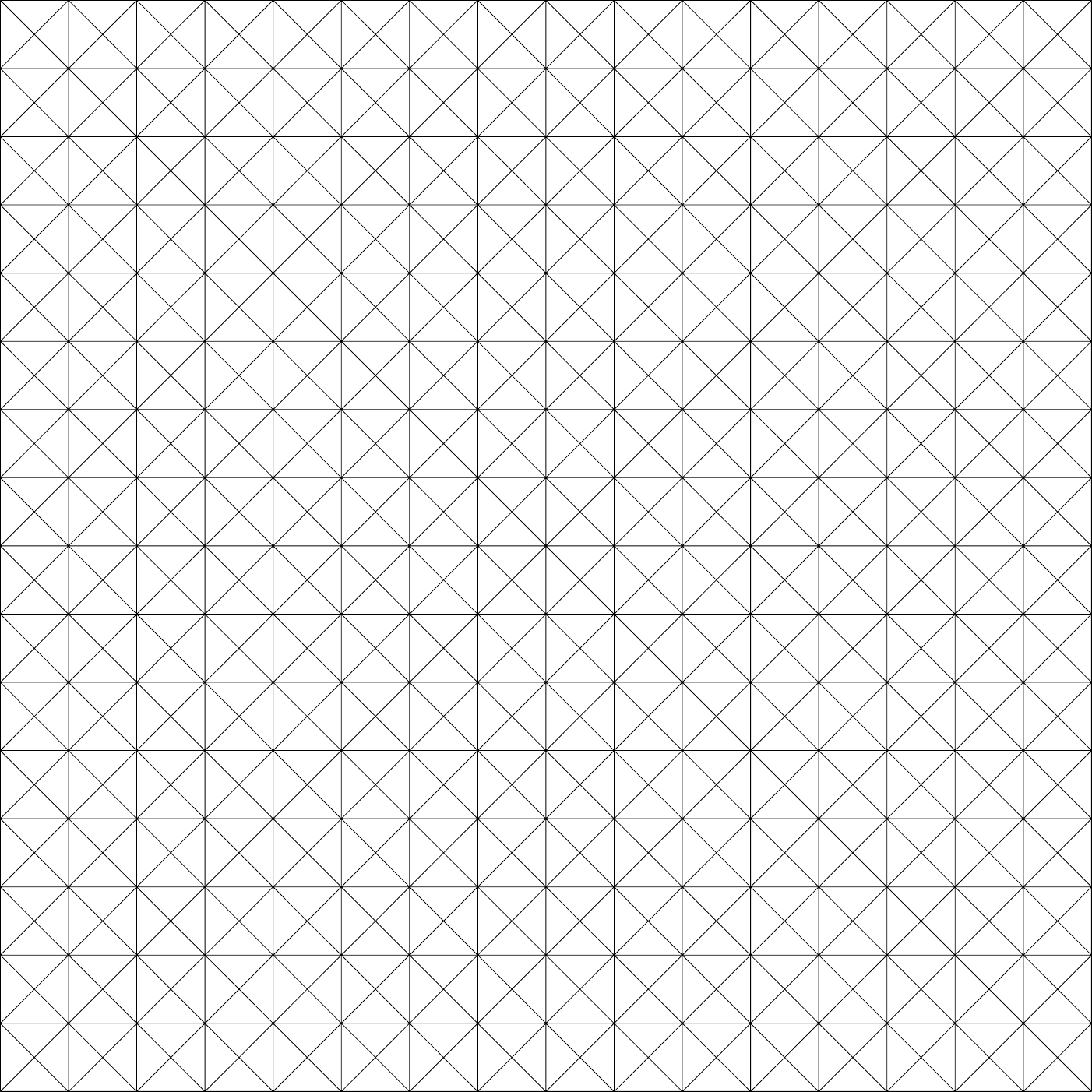}
        \hfill
        \includegraphics[width = .4 \linewidth]{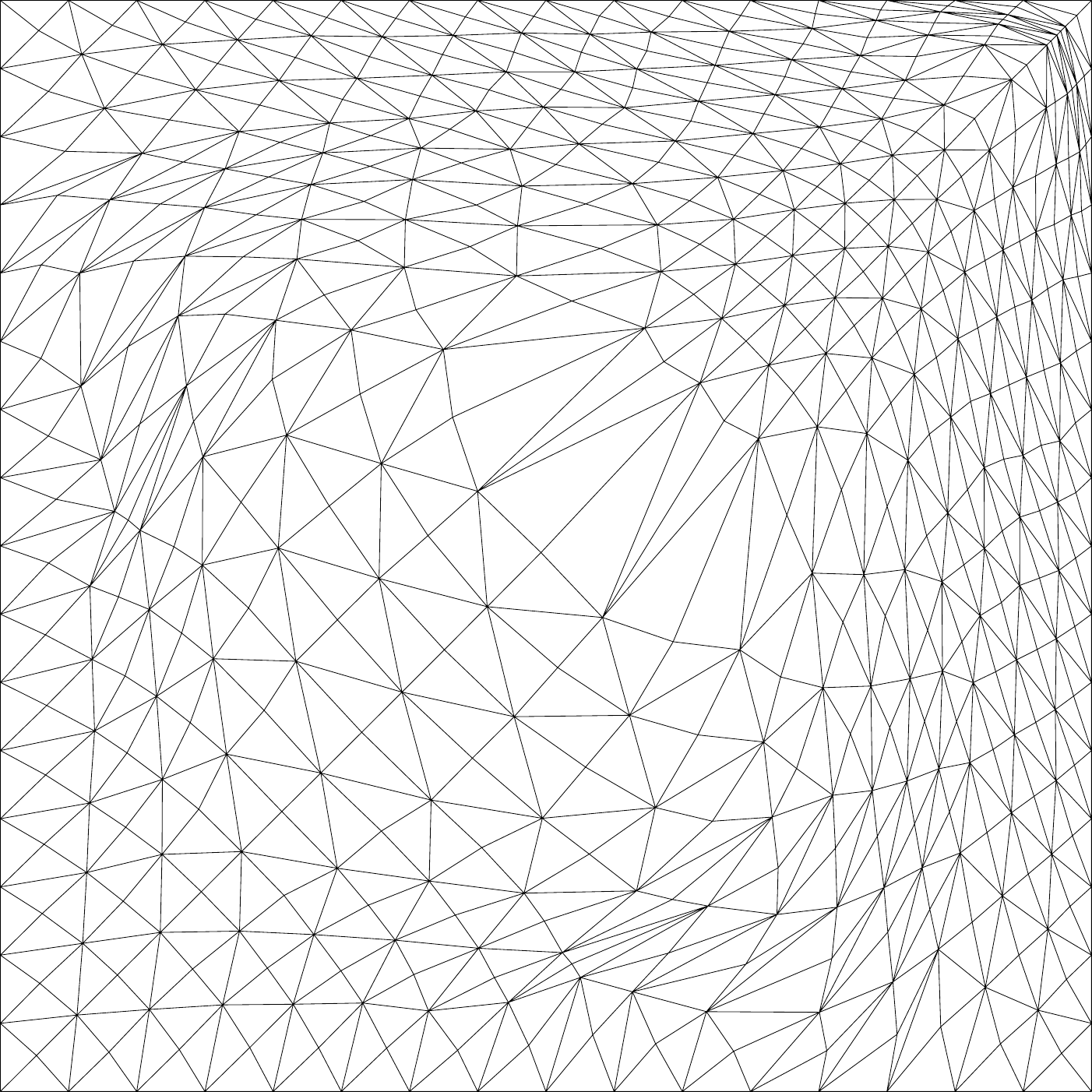}
        \\
        [.25cm]
        \includegraphics[width = .4 \linewidth]{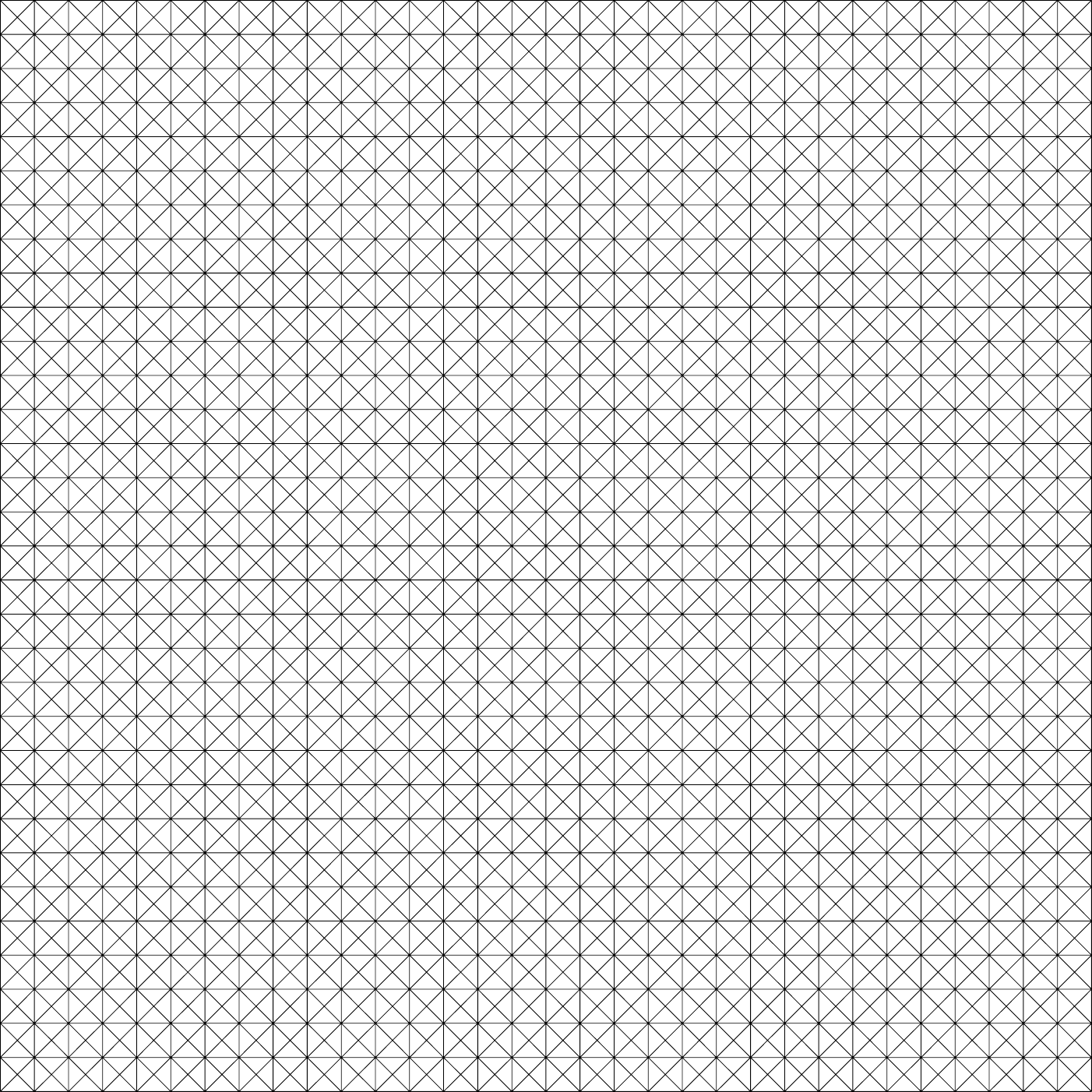}
        \hfill
        \includegraphics[width = .4 \linewidth]{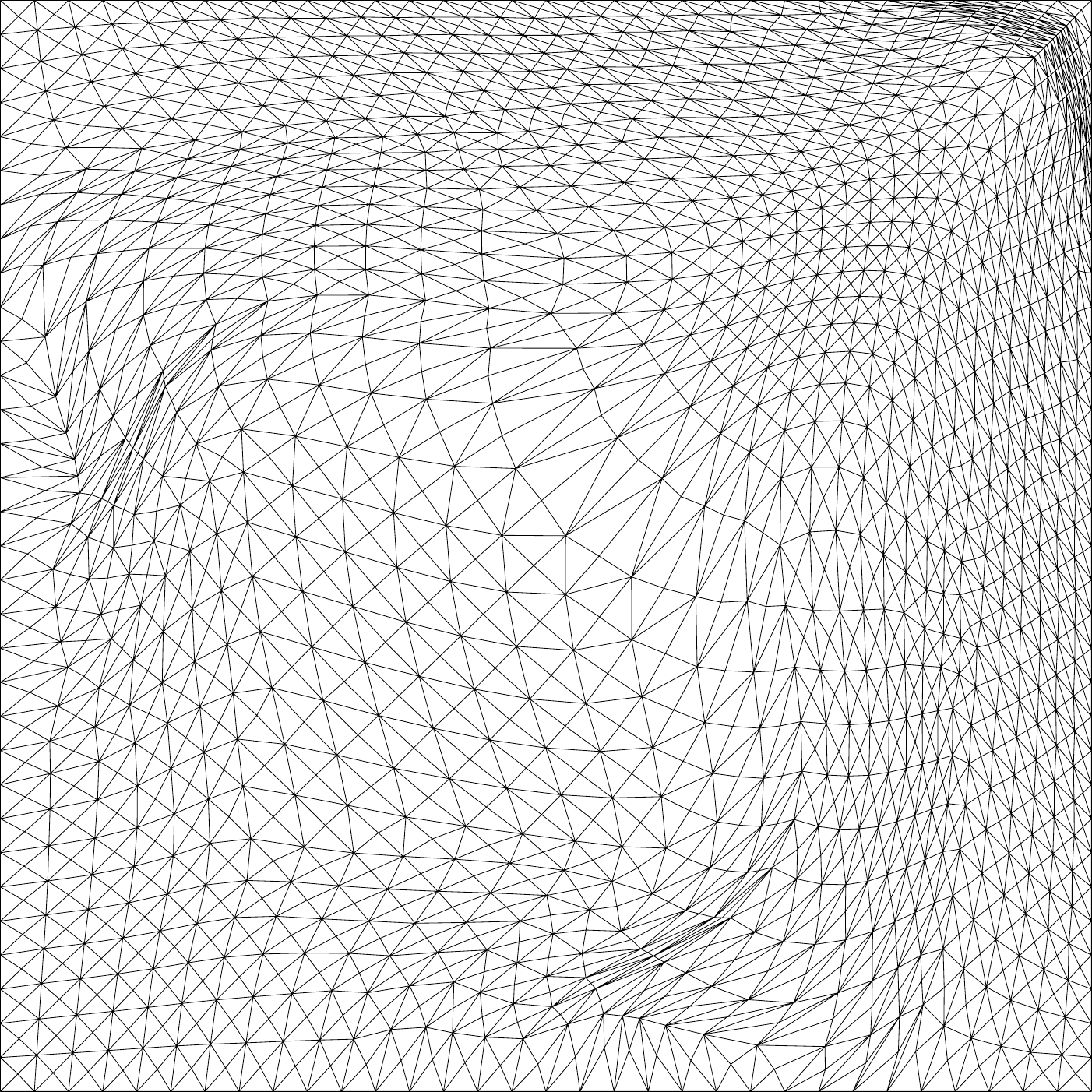}
        \caption{For a sequence of meshes, on the left is the initial uniform mesh and on the right is the optimised meshes when using Algorithm \ref{alg:ideal} with the error and linear finite elements as described in Section \ref{sec:exp:conv:ideal}.}\label{fig:exp:conv:ideal:twoD:meshes}
    \end{figure}

\subsection{Comparison between the one dimensional methods}\label{sec:exp:compare}
    It is worthwhile comparing the two methods of the residual estimator and the exact error in the one dimensional setting where it is possible.
    While this is done, it is worthwhile to also compare to a $h$-refinement strategy.
    Here, as in Section \ref{sec:exp:conv:res}, consider $f(x) = -2 C x^2(6-20x-15x^2)$ which leads to the solution $y = C(x-1)^2 x^4$.
    The meshes generated by $r$-refinement use the residual estimator and error with Algorithm \ref{alg:ideal} as described earlier in this section.
    For the $h$-refinement, this is done in a greedy way; starting from the mesh of two equal-sized elements, one calculates $\|\nabla(y- y_\Phi)\|_{L^2(T)}^2$ on each $T \in \mathcal{T}_\Phi$ and marks an element with the largest error for refinement into two elements.
    This is repeated to find the desired number of degrees of freedom.
    It is noted that this is not the best one can do, as it might be relevant to coarsen the mesh.
    
    A comparison of the meshes for the three different methods is given in Figure \ref{fig:exp:compare:meshes}.
    In Figure \ref{fig:exp:compare:errors}, the resulting errors are plotted, one might notice that all methods are performing comparably.
    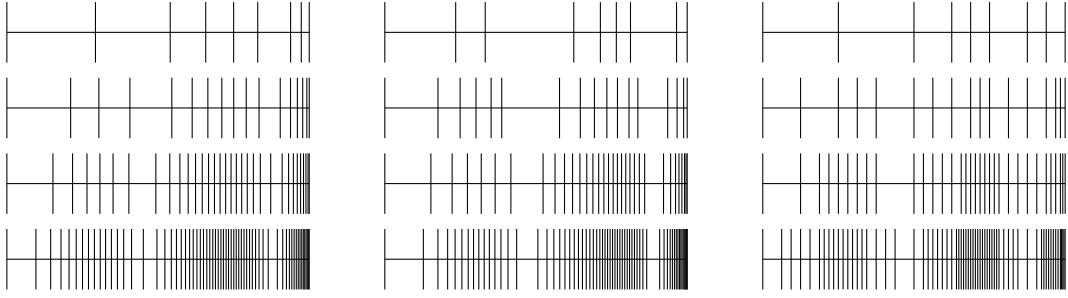
\begin{figure}
        \centering
        \begin{tikzpicture}[scale = 4]
            
            \pgfplotstableread[col sep=comma]{tikz/r_ref/oneD/residual/ref_3/final_mesh.csv}\mesh
            \draw (0,0)-- (1,0);
            \pgfplotstableforeachcolumnelement{x}\of\mesh\as\x{\draw (\x,0.1) -- (\x,-0.1);}

            \pgfplotstableread[col sep=comma]{tikz/r_ref/oneD/residual/ref_4/final_mesh.csv}\mesh
            \draw (0,-0.25)-- (1,-0.25);
            \pgfplotstableforeachcolumnelement{x}\of\mesh\as\x{\draw (\x,-0.15) -- (\x,-0.35);}

            \pgfplotstableread[col sep=comma]{tikz/r_ref/oneD/residual/ref_5/final_mesh.csv}\mesh
            \draw (0,-0.5)-- (1,-0.5);
            \pgfplotstableforeachcolumnelement{x}\of\mesh\as\x{\draw (\x,-0.4) -- (\x,-0.6);}

            \pgfplotstableread[col sep=comma]{tikz/r_ref/oneD/residual/ref_6/final_mesh.csv}\mesh
            \draw (0,-0.75)-- (1,-0.75);
            \pgfplotstableforeachcolumnelement{x}\of\mesh\as\x{\draw (\x,-0.65) -- (\x,-0.85);}

            \pgfplotstableread[col sep=comma]{tikz/r_ref/oneD/exact/ref_3/final_mesh.csv}\mesh
            \draw (1.25,0)-- (2.25,0);
            \pgfplotstableforeachcolumnelement{x}\of\mesh\as\x{\draw (1.25+\x,0.1) -- (1.25+\x,-0.1);}

            \pgfplotstableread[col sep=comma]{tikz/r_ref/oneD/exact/ref_4/final_mesh.csv}\mesh
            \draw (1.25,-0.25)-- (2.25,-0.25);
            \pgfplotstableforeachcolumnelement{x}\of\mesh\as\x{\draw (1.25+\x,-0.15) -- (1.25+\x,-0.35);}

            \pgfplotstableread[col sep=comma]{tikz/r_ref/oneD/exact/ref_5/final_mesh.csv}\mesh
            \draw (1.25,-0.5)-- (2.25,-0.5);
            \pgfplotstableforeachcolumnelement{x}\of\mesh\as\x{\draw (1.25+\x,-0.4) -- (1.25+\x,-0.6);}

            \pgfplotstableread[col sep=comma]{tikz/r_ref/oneD/exact/ref_6/final_mesh.csv}\mesh
            \draw (1.25,-0.75)-- (2.25,-0.75);
            \pgfplotstableforeachcolumnelement{x}\of\mesh\as\x{\draw (1.25+\x,-0.65) -- (1.25+\x,-0.85);}

            \pgfplotstableread[col sep=comma]{tikz/h_ref/mesh9.csv}\mesh
            \draw (2.5,0)-- (3.5,0);
            \pgfplotstableforeachcolumnelement{x}\of\mesh\as\x{\draw (2.5+\x,0.1) -- (2.5+\x,-0.1);}

            \pgfplotstableread[col sep=comma]{tikz/h_ref/mesh17.csv}\mesh
            \draw (2.5,-0.25)-- (3.5,-0.25);
            \pgfplotstableforeachcolumnelement{x}\of\mesh\as\x{\draw (2.5+\x,-0.15) -- (2.5+\x,-0.35);}

            \pgfplotstableread[col sep=comma]{tikz/h_ref/mesh33.csv}\mesh
            \draw (2.5,-0.5)-- (3.5,-0.5);
            \pgfplotstableforeachcolumnelement{x}\of\mesh\as\x{\draw (2.5+\x,-0.4) -- (2.5+\x,-0.6);}

            \pgfplotstableread[col sep=comma]{tikz/h_ref/mesh65.csv}\mesh
            \draw (2.5,-0.75)-- (3.5,-0.75);
            \pgfplotstableforeachcolumnelement{x}\of\mesh\as\x{\draw (2.5+\x,-0.65) -- (2.5+\x,-0.85);}

        \end{tikzpicture}
        \caption{
            For a sequence of mesh sizes with linear finite elements, the optimised meshes are provided as described in Section \ref{sec:exp:compare}.
            On the left is the mesh resulting from Algorithm \ref{alg:ideal} with residual estimator,
            in the middle is the mesh resulting from Algorithm \ref{alg:ideal} with the error,
            on the right is the mesh resulting from the greedy $h$-refinement as described in Section \ref{sec:exp:compare}.
            }\label{fig:exp:compare:meshes}
    \end{figure}
    \begin{figure}
        \centering
        \begin{tikzpicture}
            \pgfplotstableread[col sep=comma]{tikz/oneD_all.csv}\compareTable
            \begin{axis}[ymode = log, xmode = log, xlabel = {Degrees of freedom}, ylabel = {Error}]
                \addplot table [x expr = 2^(\thisrow{refine}) + 1, y = uniform]{\compareTable};
                \addlegendentry{Uniform mesh}
                \addplot table [x expr = 2^(\thisrow{refine}) + 1, y = h_ref]{\compareTable};
                \addlegendentry{$h$-refinement}
                \addplot table [x expr = 2^(\thisrow{refine}) + 1, y = r_ref_res]{\compareTable};
                \addlegendentry{Optimised with estimator}
                \addplot table [x expr = 2^(\thisrow{refine}) + 1, y = r_ref_ex]{\compareTable};
                \addlegendentry{Optimised with error}
            \end{axis}
        \end{tikzpicture}
        \caption{Plotted is the errors for the comparison of the one-dimensional example with linear finite elements on the uniform mesh and the meshes optimised using Algorithm \ref{alg:ideal} with the error, the same algorithm with the residual estimator, and the greedy $h$-refinement described in Section \ref{sec:exp:compare}.
        One sees that each of these methods is outperforming the uniform mesh.}\label{fig:exp:compare:errors}
    \end{figure}
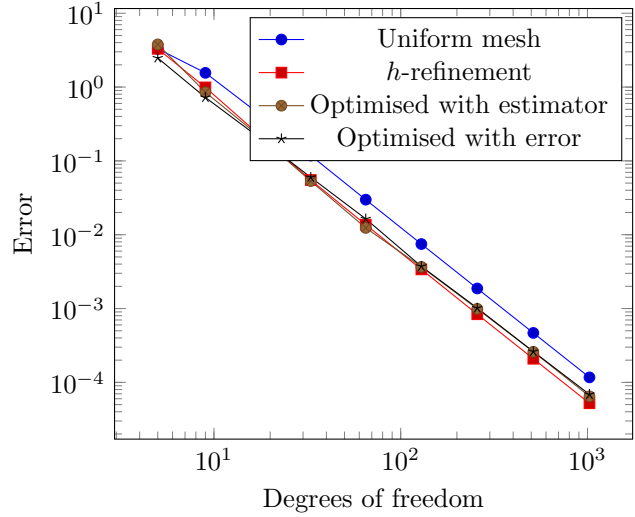

\section{Conclusion}
    This work has introduced a new strategy for $r$-refinement, based on estimators and shape optimisation.
    For the simple case of the Poisson equation, this work has shown that an idealised method converges and in a further simplified setting of one-dimension, that a fully practical algorithm converges also.

    Further work will seek to consider more general equations, as well as the extension to higher dimensions.
\section*{Acknowledgements} The author would like to thank Omar Lakkis for many fruitful discussions in the preparation of this article.
The author would also like to thank Andreas Dedner for helpful solutions to Finite Element challenges.
\section*{Data Availability Statement} There is no data associated with this work.
\printbibliography
\end{document}